\documentclass[a4paper]{amsart}

\usepackage{fixltx2e}
\usepackage[T1]{fontenc}
\usepackage[british]{babel}
\usepackage{graphicx}
\usepackage{amsfonts}
\usepackage{amssymb}
\usepackage{amsthm}
\usepackage{amsmath}
\usepackage[all]{xy}
\usepackage{hyperref}
\usepackage{mathbbol}
\usepackage{mathabx}

\theoremstyle{plain}
\newtheorem{theorem}[subsection]{Theorem}
\newtheorem{lemma}[subsection]{Lemma}
\newtheorem{corollary}[subsection]{Corollary}

\theoremstyle{definition}
\newtheorem{definition}[subsection]{Definition}
\newtheorem{remark}[subsection]{Remark}

\newcommand{\CC}{\ensuremath{\mathbb{C}}}

\newcommand{\XX}{\ensuremath{\mathbb{X}}}
\newcommand{\VV}{\ensuremath{\mathbb{V}}}

\newcommand{\Set}{\ensuremath{\mathsf{Set}}}

\newcommand{\T}{\ensuremath{\mathrm{T}}}

\newcommand{\defn}{\textbf}
\newcommand{\op}{\mathrm{op}}

\newcommand{\noproof}{\hfill \qed}

\renewcommand{\to}{\rightarrow}

\newcommand{\map}[2]{\protect{#1\to #2}}
\newcommand{\mono}[2]{\protect{#1\rightarrowtail #2}}
\newcommand{\regepi}[2]{\protect{#1\twoheadrightarrow #2}}

\newcommand{\threesom}[3]{\protect{\left[\begin{smallmatrix} {#1}\\ {#2}\\ {#3}\end{smallmatrix}\right]}}

\def\pullback{
 \ar@{-}[]+R+<6pt,-1pt>;[]+RD+<6pt,-6pt>%
 \ar@{-}[]+D+<1pt,-6pt>;[]+RD+<6pt,-6pt>}

\begin{document}

\newdir{>>}{{}*!/3.5pt/:(1,-.2)@^{>}*!/3.5pt/:(1,+.2)@_{>}*!/7pt/:(1,-.2)@^{>}*!/7pt/:(1,+.2)@_{>}}
\newdir{ >>}{{}*!/8pt/@{|}*!/3.5pt/:(1,-.2)@^{>}*!/3.5pt/:(1,+.2)@_{>}}
\newdir{ |>}{{}*!/-3.5pt/@{|}*!/-8pt/:(1,-.2)@^{>}*!/-8pt/:(1,+.2)@_{>}}
\newdir{ >}{{}*!/-8pt/@{>}}
\newdir{>}{{}*:(1,-.2)@^{>}*:(1,+.2)@_{>}}
\newdir{<}{{}*:(1,+.2)@^{<}*:(1,-.2)@_{<}}

\title[Approximate Hagemann--Mitschke co-operations]%
{Approximate Hagemann--Mitschke co-operations}

\dedicatory{Dedicated to George Janelidze on the occasion of his sixtieth birthday}

\author{Diana Rodelo}
\author{Tim Van~der Linden}

\thanks{The first author was supported by CMUC/FCT (Portugal) and the FCT Grant PTDC/MAT/120222/2010 through the European program COMPETE/FEDER}
\thanks{The second author works as \emph{charg\'e de recherches} for Fonds de la Recherche Scientifique--FNRS and would like to thank CMUC for its kind hospitality during his stays in Coimbra.}

\email{drodelo@ualg.pt}
\email{tim.vanderlinden@uclouvain.be}

\address{Departamento de Matem\'atica, Faculdade de Ci\^{e}ncias e Tecnologia, Universidade do Algarve, Campus de
Gambelas, 8005--139 Faro, Portugal}
\address{CMUC, Universidade de Coimbra, 3001--454 Coimbra, Portugal}
\address{Institut de recherche en math\'ematique et physique, Universit\'e catholique de Louvain, chemin du cyclotron~2 bte~L7.01.02, B--1348 Louvain-la-Neuve, Belgium}



\begin{abstract}
We show that varietal techniques based on the existence of operations of a certain arity can be extended to $n$-permutable categories with binary coproducts. This is achieved via what we call \emph{approximate Hagemann--Mitschke co-operations}, a generalisation of the notion of approximate Mal'tsev co-opera\-tion~\cite{AMO}. In particular, we extend characterisation theorems for $n$-permutable varieties due to J.~Hagemann and A.~Mitschke~\cite{Hagemann, Hagemann-Mitschke} to regular categories with binary coproducts.
\end{abstract}

\date{\today}

\maketitle

\section{Introduction}

A variety of universal algebras is called \emph{$n$-permutable} when its congruence relations satisfy the \emph{$n$-permutability condition}: for congruences $R$ and $S$ on an algebra~$X$, the equality $(R,S)_n=(S,R)_n$ holds, where $(R,S)_n=RSRS\cdots$ denotes the composition of $n$ alternating factors $R$ and $S$. In a categorical context, this notion was first considered by A.~Carboni, G.~M.~Kelly and M.~C.~Pedicchio in the article~\cite{Carboni-Kelly-Pedicchio}. Here an \emph{$n$-permutable category} is defined as a regular category~\cite{Barr} in which the (effective) equivalence relations satisfy the $n$-permutability condition.

For a variety $\VV$ of universal algebras, it was shown by A.~I.~Mal'tsev in~\cite{Maltsev-Sbornik} that $2$-permutability of congruences is equivalent to the condition that the theory of $\VV$ admits a ternary operation $p$ such that $p(x,y,y)=x$ and $p(x,x,y)=y$. Then $\VV$ is called a \emph{Mal'tsev variety}~\cite{Smith} or a \emph{$2$-permutable variety} and $p$ a \emph{Mal'tsev operation}. Similarly, for the strictly weaker $3$-permutability condition~\cite{Mitschke}, the theory admits ternary operations $r$ and $s$ such that $r(x,y,y)=x$, $r(x,x,y)=s(x,y,y)$ and $s(x,x,y)=y$. More generally, the $n$-permutability of congruences can be characterised by the existence of ternary operations satisfying suitable equations (\cite{Hagemann-Mitschke}, see Theorem~\ref{2 of Hagemann-Mitschke} below) or, equivalently, by the existence of certain $(n+1)$-ary operations (\cite{Groetzer, Hagemann-Mitschke, Schmidt, Wille}, see Theorem~\ref{1 of Hagemann-Mitschke} below).

The first aim of this work is to give a categorical version of such ternary (and $(n+1)$-ary) operations for $n$-permutable categories. We do this in the context of regular categories with binary coproducts via \emph{approximate Hagemann--Mitschke co-operations} with a certain \emph{approximation} (see Definition~\ref{approximate ops} and Figure~\ref{ternary co-ops}). This method extends D.~Bourn and Z.~Janelidze's approach to Mal'tsev categories via \emph{approximate Mal'tsev (co-)operations}~\cite{AMO}, which makes it possible to lift varietal techniques to the categorical level and obtain general versions of the characterisation theorems for $n$-permutable varieties mentioned above (Theorems~\ref{imaginary n-permutable 2} and~\ref{imaginary n-permutable 1}). We believe this aspect of our work gives a good illustration of the strength and generality of D.~Bourn and Z.~Janelidze's technique. See also~\cite{DB-ZJ-2009, DB-ZJ-2009b, ZJanCPIRVI, DB-ZJ-2011} where the authors further develop their theory of approximate operations.

The second aim of our paper---in fact, the problem which we originally set out to solve---is answering the following question. J.~Hagemann discovered a purely categorical characterisation of $n$-permutable varieties~\cite{Hagemann-Mitschke}:

\begin{theorem}\label{3 of Hagemann-Mitschke}
For any variety $\VV$ of universal algebras, the following statements are equivalent:
\begin{enumerate}
\item the congruence relations of every algebra of $\VV$ are $n$-permutable;
\item for $A\in\VV$, every reflexive subalgebra $R$ of $A\times A$ satisfies $R^{\op} \leq R^{n-1}$;
\item for $A\in\VV$, every reflexive subalgebra $R$ of $A\times A$ satisfies $R^{n}\leq R^{n-1}$.\noproof
\end{enumerate}
\end{theorem}

These three conditions make sense in arbitrary regular categories; nevertheless, they are not mentioned in the article~\cite{Carboni-Kelly-Pedicchio}. For the proof the authors of~\cite{Hagemann-Mitschke} refer to the unpublished work~\cite{Hagemann}. It is indeed not difficult to find a proof which is valid in varieties of algebras (see~\cite{MFVdL2} for part of it) but we failed to produce a categorical argument.

Assuming that binary coproducts exist---in fact, finite copowers suffice---we can mimic the varietal arguments in terms of ternary or $(n+1)$-arity operations, using approximate Hagemann--Mitschke co-operations instead. This is what we do in the present paper. Thus we obtain a version of the above characterisation theorem, valid in any regular category with binary coproducts (Theorem~\ref{imaginary total thm}). This, in turn, implies that in an $n$-permutable category with binary coproducts, any reflexive and transitive relation is symmetric (Corollary~\ref{Corollary RTS}).

On the other hand, in recent work with Z.~Janelidze~\cite{JRVdL1} we prove that Theorem~\ref{3 of Hagemann-Mitschke} is actually valid in regular categories, independently of the existence of binary coproducts---using a very different approach, since it does not (and can not) involve approximate (co-)operations.

\section{Preliminaries}\label{Preliminaries}

We recall the main definitions and properties known for $n$-permutable varieties from~\cite{Hagemann-Mitschke} and for $n$-permutable categories we follow~\cite{Carboni-Kelly-Pedicchio}.

\subsection{Relations}\label{Relations}
A category $\CC$ with finite limits is called a \defn{regular} category~\cite{Barr} when every kernel pair has a coequaliser and regular epimorphisms are stable under pulling back. In a regular category any morphism $f\colon A\rightarrow B$ can be decomposed into $f=m p$, where $p$ is a regular epimorphism and $m$ is a monomorphism. Regular categories give a suitable context for composing relations.

A \defn{relation} $R$ from $A$ to $B$ is a subobject $\langle r_1,r_2\rangle\colon\mono{R}{A\times B}$. The opposite relation, denoted $R^{\op}$, is the relation from $B$ to $A$ determined by the subobject $\langle r_2,r_1\rangle\colon R\rightarrowtail B\times A$. Given another relation $S$ from $B$ to $C$, the composite relation of $R$ and $S$ is a relation, denoted $SR$, from $A$ to $C$.

Given morphisms $a\colon\map{X}{A}$ and $b\colon\map{X}{B}$, we say that $\langle a,b\rangle$ \defn{belongs} to~$R$ when there exists a morphism $\chi\colon\map{X}{R}$ such that $r_1\chi=a$ and $r_2\chi=b$; we write $\langle a,b\rangle\in R$. For any morphism $c\colon\map{X}{C}$, we have $\langle a,c\rangle\in SR$ if and only if there exists a regular epimorphism $\zeta\colon\regepi{Z}{X}$ and a morphism $x\colon\map{Z}{B}$ such that $\langle a\zeta,x\rangle\in R$ and $\langle x, c\zeta\rangle\in S$ (see Proposition~2.1 in~\cite{Carboni-Kelly-Pedicchio}). This observation trivially extends to the composite of more than two relations. Moreover, when $R'$ is another relation from $A$ to $B$, then $R\leq R'$ if and only if any pair of morphisms $\langle a, b\rangle$ that belongs to $R$ also belongs to $R'$.

A relation $R$ from an object $A$ to $A$ is simply called a \defn{relation on $A$}. We say that $R$ is \defn{reflexive} when $1_A\leq R$, \defn{symmetric} when $R^{\op} = R$ and \defn{transitive} when $RR = R$. As usual, a relation $R$ on $A$ is an \defn{equivalence relation} when it is reflexive, symmetric and transitive. In particular, the kernel relation $\langle f_1,f_2\rangle\colon\mono{A\times_{B}A}{A\times A}$ of a morphism $f\colon\map{A}{B}$ is called an \defn{effective} equivalence relation or \defn{congruence}.

\subsection{$n$-Permutable varieties~\cite{Hagemann-Mitschke}}\label{$n$-permutable varieties}
A \defn{Mal'tsev} (or \defn{$2$-permutable}) variety of universal algebras is such that the composition of congruences is $2$-permutable, i.e., $RS=SR$, for any pair of congruences $R$ and $S$ on the same object. \defn{$3$-permutable} varieties satisfy the strictly
weaker $3$-permutability condition: $RSR= SRS$. More generally, for $n\geq 2$, \defn{$n$-permutable} varieties satisfy the $n$-permutability condition $(R,S)_n=(S,R)_n$, where $(R,S)_n=RSRS\cdots$ denotes the composite of~$n$ alternating factors $R$ and $S$. We write $R^{n}=(R,R)_{n}$ for the $n$-th power of $R$.

It is well known that an $n$-permutable variety of universal algebras is characterised by the condition that its theory contains $n-1$ ternary or, equivalently, $n+1$ operations of arity $n+1$ satisfying appropriate identities:

\begin{theorem}[Theorem~2 of~\cite{Hagemann-Mitschke}]\label{2 of Hagemann-Mitschke}
For any variety $\VV$ of universal algebras, the following statements are equivalent:
\begin{enumerate}
\item the congruence relations of every algebra of $\VV$ are $n$-permutable;
\item there exist ternary algebraic operations $w_1$, \dots, $w_{n-1}$ of $\VV$ for which the identities
\[
\begin{cases}
 w_1(x,y,y)=x,\\
 w_{i}(x,x,y) = w_{i+1}(x,y,y), & \text{for $i\in\{1, \dots, n-2\}$,}\\
 w_{n-1}(x,x,y)=y
\end{cases}
\]
hold.\noproof
\end{enumerate}
\end{theorem}

\begin{theorem}[\cite{Wille, Schmidt, Hagemann-Mitschke}]\label{1 of Hagemann-Mitschke}
For any variety $\VV$ of universal algebras, the following statements are equivalent:
\begin{enumerate}
\item the congruence relations of every algebra of $\VV$ are $n$-permutable;
\item there exist $(n+1)$-ary algebraic operations $v_0$, \dots, $v_n$ of $\VV$ for which the identities
\[
\begin{cases}
 v_0(x_0, \dots, x_n)=x_0,\\
 v_{i-1}(x_0,x_0, x_2, x_2, \dots ) = v_i(x_0,x_0, x_2, x_2, \dots ), &\text{$i$ even,}\\
 v_{i-1}(x_0,x_1, x_1, x_3, x_3, \dots ) = v_i(x_0,x_1, x_1, x_3, x_3, \dots ), & \text{$i$ odd,}\\
 v_n(x_0, \dots, x_n)=x_n
\end{cases}
\]
hold.\noproof
\end{enumerate}
\end{theorem}

In particular, a Mal'tsev variety has a \defn{Mal'tsev operation} $p$ which satisfies
\begin{equation*}\label{Mal'tsev alg p}
\begin{cases}
 p(x,y,y)=x,\\
 p(x,x,y)=y.
\end{cases}
\end{equation*}
A $3$-permutable variety can be characterised by the existence of two ternary operations, $r$ and~$s$, satisfying the left hand side identities
\begin{equation*}\label{Goursat alg r,s}
\begin{cases}
 r(x,y,y)=x,\\
 r(x,x,y)=s(x,y,y),\\
 s(x,x,y)=y
\end{cases}
\qquad\qquad\qquad
\begin{cases}
 p(x,y,y,z)=x,\\
 p(x,x,y,y)=q(x,x,y,y),\\
 q(x,y,y,z)=z
\end{cases}
\end{equation*}
or, equivalently, by the existence of two quaternary operations, $p$ and~$q$, such that the identities on the right above hold.

\begin{remark}\label{Goursat Remark}\label{(n+1)-ary vs ternary}
The equivalence between the existence of ternary operations and the existence of quaternary operations is given by the identities
\[
p(x,y,z,w)=r(x,y,z)\qquad \text{and}\qquad q(x,y,z,w)=s(y,z,w)
\]
on the one hand,
\[
r(x,y,z)=p(x,y,z,z)\qquad \text{and}\qquad s(x,y,z)=q(x,x,y,z)
\]
on the other. More generally~\cite{Hagemann-Mitschke}, the equivalence between the existence of ternary and the existence of $(n+1)$-ary operations for $n$-permutable varieties is given by the identities
\[
\begin{cases}
v_0(x_0, \dots, x_n)=x_0,\\
v_i(x_0, \dots, x_n)=w_i(x_{i-1},x_i,x_{i+1}),\;\; \text{for $i\in\{1, \dots, n-1\}$,}\\
v_n(x_0, \dots, x_n)=x_n
\end{cases}
\]
and $w_i(x,y,z)=v_i(\underset{i}{\underbrace{x, \dots, x}}, y,\underset{n-i}{\underbrace{z, \dots, z}})$ for $i\in\{1, \dots, n-1\}$.
\end{remark}

As mentioned in the introduction (Theorem~\ref{3 of Hagemann-Mitschke}), J.~Hagemann also obtained alternative characterisations which involve conditions on reflexive relations.

\subsection{$n$-Permutable categories~\cite{Carboni-Kelly-Pedicchio}}\label{n-permutable categories}
A regular category is an \defn{$n$-per\-mu\-ta\-ble category} when the composition of (effective) equivalence relations on a given object is $n$-permutable: for two (effective) equivalence relations $R$ and~$S$ on the same object, we have $(R,S)_n=(S,R)_n$. In fact, it suffices to have one of the inequalities, say $(R,S)_n\leq (S,R)_n$. These categories can be characterised by the condition that, for every reflexive relation $E$, the relation $(E,E^{\op})_{n-1}$ is an equivalence relation or, equivalently, a transitive relation.

\section{Main result}\label{Main result}

In this section we prove a categorical version of Theorem~\ref{2 of Hagemann-Mitschke}, valid in regular categories with binary coproducts. We shall repeatedly use the techniques from Subsection~\ref{Relations} without further mention.

\begin{definition}\label{approximate ops}
Let $\CC$ be a category with binary products. We say that morphisms $w_1$, \dots, $w_{n-1}\colon\map{X^{3}}{A}$ are \defn{approximate Hagemann--Mitschke operations} (on $X$) \defn{with approximation} $\alpha\colon\map{X}{A}$ if the diagram in Figure~\ref{ops} commutes.
\begin{figure}[h]
\[
\vcenter{\xymatrix@!0@R=3em@C=2.2em{
 &&&&& X\ar@{.>}[dd]^-{\alpha}\\
 X^{2}\ar[rrrrru]^-{\pi_1}\ar[dd]_-{1_X\times\Delta_X} & & & & & & & & & & X^{2}\ar[lllllu]_-{\pi_{2}}\ar[dd]^-{\Delta_X\times 1_X}\\
 &&&&& A\\
 X^{3}\ar@{.>}[rrrrru]^-{w_1} & & & & & & & & & & X^{3}\ar@{.>}[lllllu]_-{w_{n-1}} \\
 X^2 \ar[u]^-{\Delta_X\times 1_X} \ar[rr]_-{1_X\times\Delta_X} & & X^{3}\ar@{.>}[rrruu]^-{w_2} &&&&&&
 X^{3} \ar@{.>}[llluu]_-{w_{n-2}} && X^2 \ar[u]_-{1_X\times\Delta_X} \ar[ll]^-{\Delta_X\times 1_X}\\
 && X^{2} \ar[u] \ar[rr] && X^{3} \ar@{.>}[ruuu]^-{w_3} \ar@{}[rrrru]|-{\rotatebox{20}{$\cdots$}}}}
\]
\caption{Approximate Hagemann--Mitschke operations}\label{ops}
\end{figure}
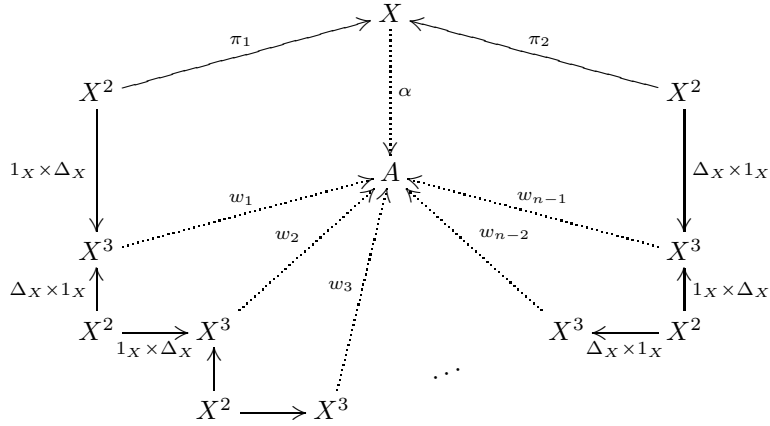

If $(A,\alpha,w_1,\dots,w_{n-1})$ is the colimit of the outer solid diagram, then $w_1$, \dots, $w_{n-1}$ are called \defn{universal} approximate Hagemann--Mitschke operations with approximation $\alpha$.
\end{definition}

The approximation $\alpha$ is uniquely determined by any of its operations $w_i$ since we have $\alpha=w_i\langle 1_X,\dots,1_X\rangle$. To say that $w_1$, \dots, $w_{n-1}$ are approximate Hagemann--Mitschke operations with approximation $\alpha$ is equivalent to having, for every object~$W$ and all morphisms $x_0$, \dots, $x_n\colon\map{W}{X}$, identities which correspond to those given in Theorem~\ref{1 of Hagemann-Mitschke}.

When $\CC$ is a category with binary products and finite colimits, then there always exist universal approximate Hagemann--Mitschke operations given by the colimit of the outer diagram of Figure~\ref{ops}.

If $A=X$ and $\alpha = 1_{X}$ then the $w_1$, \dots, $w_{n-1}$ are ``real'' operations on $X$ which provide~$X$ with an internal structure. For instance, in $\Set$ this means that $(X,w_1,\dots,w_n)$ is a kind of generalised Mal'tsev algebra.

We work in the dual category $\CC^{\op}=\XX^{\XX}$, where $\XX$ is a category with finite limits and binary coproducts. So, (universal) approximate Hagemann--Mitschke operations $w_1$, \dots, $w_{n-1}\colon\map{X^{3}}{A}$ with approximation $\alpha\colon\map{X}{A}$ in $\CC$ are, in fact, \defn{(universal) approximate Hagemann--Mitschke co-operations} $w_1$, \dots, $w_{n-1}\colon\map{A}{3X}$ with approximation $\alpha\colon\map{A}{X}$ in $\CC^{\op}$. We consider the particular case when ${X=1_{\XX}}$. Again, universal approximate co-operations always exist: $w_1$, \dots, $w_{n-1}$ and $\alpha$ are natural transformations defined, for each object $X$ in $\XX$, by the limit of one of the outer solid diagram in Figure~\ref{ternary co-ops}.
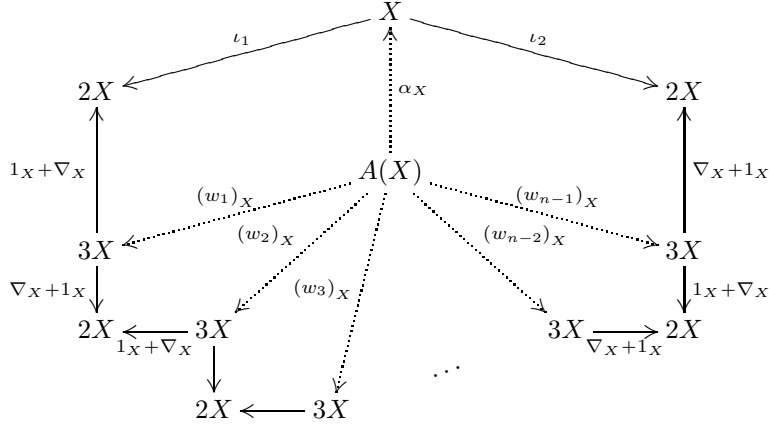
\begin{figure}[h]
\[
\vcenter{\xymatrix@!0@R=3em@C=2.2em{
 &&&&& X \ar[llllld]_-{\iota_1} \ar[rrrrrd]^-{\iota_2} \\
 2X & & & & & & & & & & 2X\\
 &&&&& A(X) \ar@{.>}[uu]_-{\alpha_X} \ar@{.>}[llllld]_-{{(w_1)}_X} \ar@{.>}[llldd]_-{{(w_2)}_X} \ar@{.>}[lddd]_-{{(w_3)}_X}
 \ar@{.>}[rrrdd]^-{{(w_{n-2})}_X} \ar@{.>}[rrrrrd]^-{{(w_{n-1})}_X} \\
 3X \ar[uu]^-{1_X+\nabla_X} \ar[d]_-{\nabla_X+1_X} & & & & & & & & & &
 3X \ar[uu]_-{\nabla_X+ 1_X} \ar[d]^-{1_X+\nabla_X} \\
 2X & & 3X \ar[ll]^-{1_X+\nabla_X} \ar[d] &&&&&&
 3X \ar[rr]_-{\nabla_X+1_X} && 2X \\
 && 2X && 3X \ar[ll] \ar@{}[rrrru]|-{\rotatebox{20}{$\cdots$}} }}
\]
\caption{Approximate Hagemann--Mitschke co-operations}\label{ternary co-ops}
\end{figure}

Now we obtain the claimed categorical version of Theorem~\ref{2 of Hagemann-Mitschke}:

\begin{theorem}\label{imaginary n-permutable 2}
Let $\XX$ be a regular category with binary coproducts. The following statements are equivalent:
\begin{enumerate}
\item the approximation $\alpha\colon A\Rightarrow 1_{\XX}$ of the universal approximate Hagemann--Mitschke co-operations on $1_{\XX}$ is a natural transformation, all of whose components are regular epimorphisms in $\XX$;
\item there exist approximate Hagemann--Mitschke co-operations on $1_{\XX}$ such that the approximation $\alpha\colon A\Rightarrow 1_{\XX}$ is a natural transformation, all of whose components are regular epimorphisms in $\XX$;
\item $\XX$ is an $n$-permutable category.
\end{enumerate}
\end{theorem}
\begin{proof}
(a) $\Rightarrow$ (b) is obvious.

(b) $\Rightarrow$ (c) Let $R$ and $S$ be equivalence relations on an object~$Y$. We must prove that $(R,S)_n\leq (S,R)_n$. Let $a$, $b\colon\map{X}{Y}$ be morphisms such that $\langle a,b\rangle\in (R,S)_n$. Then there exists a regular epimorphism $\zeta\colon\regepi{Z}{X}$ together with morphisms $x_1$, \dots, $x_{n-1}\colon\map{Z}{Y}$ such that
\[
\langle a\zeta, x_1\rangle\in R,\;
\langle x_1,x_2\rangle\in S,\;
\dots,\;
\langle x_{n-2},x_{n-1} \rangle\in S\; \text{and}\; \langle x_{n-1}, b\zeta\rangle\in R,
\]
if $n$ is odd. Then the relation $S$ contains the couples $\langle a\zeta, a\zeta\rangle$, $\langle x_1, x_1\rangle$, $\langle x_1,x_2 \rangle$, $\langle x_3,x_3 \rangle$, $\langle x_3,x_4 \rangle$ \dots, $\langle x_{n-4},x_{n-4} \rangle$, $\langle x_{n-4},x_{n-3} \rangle$, $\langle x_{n-2},x_{n-2} \rangle$, $\langle x_{n-2},x_{n-1} \rangle$ and $\langle b\zeta, b\zeta\rangle$. If we compose the first triple $\langle a\zeta, a\zeta\rangle$, $\langle x_1, x_1\rangle$, $\langle x_1,x_2 \rangle$ with $(w_1)_Z$, the second triple $\langle x_1, x_1\rangle$, $\langle x_1,x_2 \rangle$, $\langle x_3,x_3 \rangle$ with $(w_2)_Z$, and so on, we get
\[
\begin{cases}
 \left\langle a\zeta \alpha_Z, \threesom{a\zeta}{x_1}{x_2}{(w_1)}_Z \right\rangle \in S \vspace{3pt}\\
 \left\langle \threesom{x_1}{x_2}{x_3}{(w_2)}_Z, \threesom{x_2}{x_3}{x_4}{(w_3)}_Z \right\rangle \in SS^{\op}=S \vspace{3pt}\\
\qquad\vdots \\
 \left\langle \threesom{x_{n-4}}{x_{n-3}}{x_{n-2}}{(w_{n-3})}_Z, \threesom{x_{n-3}}{x_{n-2}}{x_{n-1}}{(w_{n-2})}_Z \right\rangle \in SS^{\op}=S \vspace{3pt}\\
 \left\langle b\zeta \alpha_Z, \threesom{x_{n-2}}{x_{n-1}}{b\zeta}{(w_{n-1})}_Z\right\rangle \in S
\end{cases}
\]
because $\left(\nabla_Z +1_Z\right){(w_j)}_Z = \left(1_Z+\nabla_Z\right) {(w_{j+1})}_Z$, for $j$ even, $j\in\{2,\dots, n-3\}$. Similarly, since $
\langle a\zeta, a\zeta\rangle$, $\langle x_1, a\zeta \rangle$, $\langle x_2,x_3 \rangle$, $\langle x_3,x_3 \rangle$, $\dots$, $\langle x_{n-3},x_{n-2} \rangle$, $\langle x_{n-2},x_{n-2} \rangle$, $\langle x_{n-1},b\zeta \rangle$ and $\langle b\zeta, b\zeta\rangle$ are all in $R$, we get
\[
\begin{cases}
 \left\langle \threesom{a\zeta}{x_1}{x_2}{(w_1)}_Z, \threesom{x_1}{x_2}{x_3}{(w_2)}_Z \right\rangle \in R^{\op}R=R \vspace{3pt}\\
\qquad\vdots \\
 \left\langle \threesom{x_{n-3}}{x_{n-2}}{x_{n-1}}{(w_{n-2})}_Z, \threesom{x_{n-2}}{x_{n-1}}{b\zeta}{(w_{n-1})}_Z \right\rangle \in R^{\op}R=R.
\end{cases}
\]
We can now conclude that $\langle a\zeta \alpha_Z,b\zeta \alpha_Z \rangle \in (S,R)_n$, which implies that $\langle a,b \rangle\in (S,R)_n$, since~$\zeta$ and $\alpha_Z$ are regular epimorphisms.

For $n$ even the proof is similar. Now we have $\langle a\zeta, x_1\rangle\in S$, $\langle x_1,x_2\rangle\in R$, \dots, $\langle x_{n-2},x_{n-1} \rangle\in S$ and $\langle x_{n-1}, b\zeta\rangle\in R$ and should consider $\langle a\zeta, a\zeta\rangle$, $\langle x_1, x_1\rangle$, $\langle x_1,x_2 \rangle$, $\langle x_3, x_3\rangle$, $\dots$, $\langle x_{n-3},x_{n-3} \rangle$, $\langle x_{n-3},x_{n-2} \rangle$, $\langle x_{n-1},x_{n-1} \rangle$, $\langle x_{n-1}, b\zeta\rangle$ $\in R$
and
$\langle a\zeta, a\zeta\rangle$, $\langle x_1, a\zeta \rangle$, $\langle x_2,x_3 \rangle$, $\langle x_3,x_3 \rangle$, $\dots$, $\langle x_{n-3},x_{n-3} \rangle$, $\langle x_{n-2},x_{n-1} \rangle$, $\langle x_{n-1},x_{n-1} \rangle$, $\langle b\zeta, b\zeta\rangle$ $\in S$
and proceed as above.

(c) $\Rightarrow$ (a) We must prove that $\alpha_X$ in the limit diagram of Figure~\ref{ternary co-ops} is a regular epimorphism for every object $X$ in $\XX$. If $n$ is odd, we take $k=\tfrac{n+1}{2}$ and let $R$ be the kernel relation of $k\nabla_X$ and $S$ the kernel relation of $1_X+(k-1)\nabla_X+1_X$, and if~$n$ is even, we take $R$ and $S$ to be the respective kernel relations of $1_X+k\nabla_X$ and $k\nabla_X+1_X$ where $k=\tfrac{n}{2}$. For the coproduct inclusions $\iota_1$, \dots, $\iota_{n+1}\colon\map{X}{(n+1)X}$ we have
\[
\langle\iota_1,\iota_{n+1}\rangle\in (R,S)_n=(S,R)_n.
\]
So there exist a regular epimorphism $\zeta\colon\regepi{Z}{X}$ as well as morphisms $x_1$, \dots, $x_{n-1}\colon\map{Z}{(n+1)X}$ such that
\[
\begin{cases}
\langle\iota_1\zeta, x_1\rangle\in S,\; \langle x_1, x_2\rangle\in R,\; \langle x_2, x_3\rangle\in S,\; \dots,\\
\quad\langle x_{n-2},x_{n-1}\rangle\in R,\; \langle x_{n-1},\iota_{n+1}\zeta\rangle \in S,\; & \text{$n$ odd}\\\\
\langle\iota_1\zeta, x_1\rangle\in R,\; \langle x_1, x_2\rangle\in S,\; \langle x_2, x_3\rangle\in R,\; \dots,\\
\quad\langle x_{n-2},x_{n-1}\rangle\in R,\; \langle x_{n-1},\iota_{n+1}\zeta\rangle \in S,\; & \text{$n$ even.}
\end{cases}
\]
The regular epimorphism $\zeta\colon \regepi{Z}{X}$ and the morphisms
\[
\left((\nabla_i)_X+1_X+(\nabla_{n-i})_X\right) x_i\colon Z\to 3X,\;\; i\in \{1, \dots, n-1\},
\]
give a cone of the outer diagram in Figure~\ref{ternary co-ops}: see Remark~\ref{(n+1)-ary vs ternary}. Here ${(\nabla_i)}_X=[1_X, \dots, 1_X ]^{\T}\colon iX\to X$; in particular, ${(\nabla_1)}_X=1_X$ and ${(\nabla_2)}_X=\nabla_X$. This guarantees the existence of a unique morphism $\lambda\colon\map{Z}{B(X)}$ such that, in particular, $\alpha_X=\zeta \lambda$. Hence $\alpha_X$ is a regular epimorphism.
\end{proof}


\section{Applications}\label{Applications}

We now extend Theorem~\ref{1 of Hagemann-Mitschke} to the context of $n$-permutable categories. As for $n$-permutable varieties (Remark~\ref{(n+1)-ary vs ternary}), we also have a correspondence between approximate Hagemann--Mitschke co-op\-er\-a\-tions and certain $(n+1)$-ary co-op\-er\-a\-tions. Similarly, (universal) approximate $(n+1)$-ary co-operations $v_1$, \dots, $v_{n-1}$ with approximation $\beta$ are natural transformations defined, for each object $X$ in $\XX$, by the (limit of the outer solid) commutative diagrams in Figure~\ref{co-ops odd}, when $n$ is odd, and in Figure~\ref{co-ops even}, when $n$ is even.
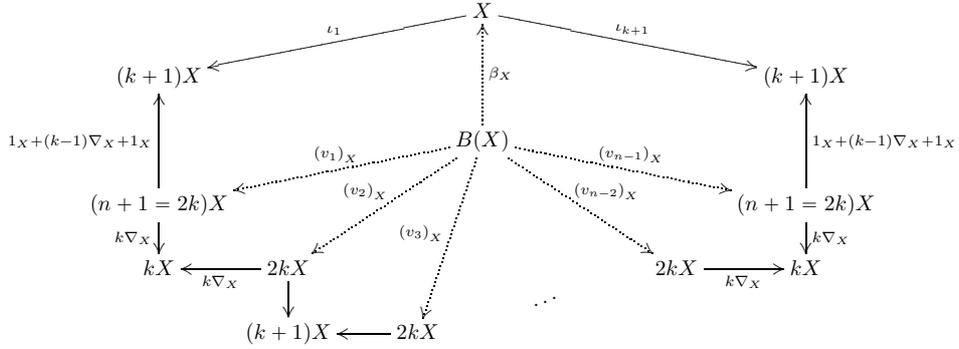
\begin{figure}[h]
\[
\resizebox{\textwidth}{!}
{\xymatrix@!0@R=3em@C=3em{
 &&&&& X \ar[llllld]_-{\iota_1} \ar[rrrrrd]^-{\iota_{k+1}} \\
 (k+1)X & & & & & & & & & & (k+1)X\\
 &&&&& B(X) \ar@{.>}[uu]_-{\beta_X} \ar@{.>}[llllld]_-{{(v_1)}_X} \ar@{.>}[llldd]_-{{(v_2)}_X} \ar@{.>}[lddd]_-{{(v_3)}_X}
 \ar@{.>}[rrrdd]^-{{(v_{n-2})}_X} \ar@{.>}[rrrrrd]^-{{(v_{n-1})}_X} \\
 (n+1=2k)X \ar[uu]^-{1_X+(k-1)\nabla_X+ 1_X} \ar[d]_-{k\nabla_X} & & & & & & & & & &
 (n+1=2k)X \ar[uu]_-{1_X+(k-1)\nabla_X+ 1_X} \ar[d]^-{k\nabla_X} \\
 kX & & 2kX \ar[ll]^-{k\nabla_X} \ar[d] &&&&&&
 2kX \ar[rr]_-{k\nabla_X} && kX \\
 && (k+1)X && 2kX \ar[ll] \ar@{}[rrrru]|-{\rotatebox{20}{$\cdots$}} }}
\]
\caption{Approximate $(n+1)$-ary co-operation, odd case (${n=2k-1}$ for $k\geq 2$)}\label{co-ops odd}
\end{figure}
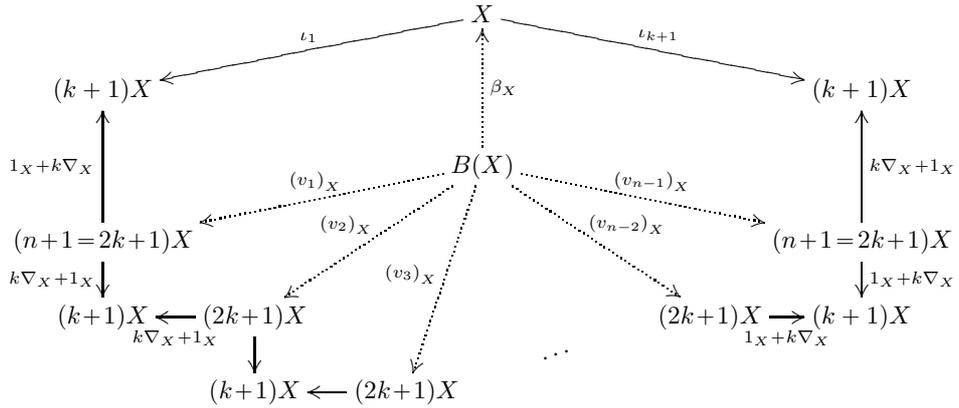
\begin{figure}[h]
\[
\resizebox{\textwidth}{!}
{\xymatrix@!0@R=3em@C=3em{
 &&&&& X \ar[llllld]_-{\iota_1} \ar[rrrrrd]^-{\iota_{k+1}} \\
 (k+1)X & & & & & & & & & & (k+1)X\\
 &&&&& B(X) \ar@{.>}[uu]_-{\beta_X} \ar@{.>}[llllld]_-{{(v_1)}_X} \ar@{.>}[llldd]_-{{(v_2)}_X} \ar@{.>}[lddd]_-{{(v_3)}_X}
 \ar@{.>}[rrrdd]^-{{(v_{n-2})}_X} \ar@{.>}[rrrrrd]^-{{(v_{n-1})}_X} \\
 (n\!+\!1\!=\!2k\!+\!1)X \ar[uu]^-{1_X+k\nabla_X} \ar[d]_-{k\nabla_X+1_X} & & & & & & & & & &
 (n\!+\!1\!=\!2k\!+\!1)X \ar[uu]_-{k\nabla_X+ 1_X} \ar[d]^-{1_X+k\nabla_X} \\
 (k\!+\!1)X & & (2k\!+\!1)X \ar[ll]^-{\stackrel{}{k\nabla_X+ 1_X}} \ar[d] &&&&&&
 (2k\!+\!1)X \ar[rr]_-{\stackrel{}{1_X+k\nabla_X}} && (k+1)X \\
 && (k\!+\!1)X && (2k\!+\!1)X \ar[ll] \ar@{}[rrrru]|-{\rotatebox{20}{$\cdots$}} }}
\]
\caption{Approximate $(n+1)$-ary co-operation, even case (${n=2k}$ for $k\geq 2$)}\label{co-ops even}
\end{figure}

\begin{theorem}\label{imaginary n-permutable 1}
Let $\XX$ be a regular category with binary coproducts. The following statements are equivalent:
\begin{enumerate}
\item the approximation $\beta\colon B\Rightarrow 1_{\XX}$ of the universal approximate $(n+1)$-ary co-operations in figures~\ref{co-ops odd} and \ref{co-ops even} is a natural transformation, all of whose components are regular epimorphisms in $\XX$;
\item there exist approximate $(n+1)$-ary co-operations as in figures~\ref{co-ops odd} and \ref{co-ops even} such that the approximation $\beta\colon B\Rightarrow 1_{\XX}$ is a natural transformation, all of whose components are regular epimorphisms in $\XX$;
\item $\XX$ is an $n$-permutable category.
\end{enumerate}
\end{theorem}
\begin{proof}
It suffices to show that condition (a) is equivalent to condition (a) from Theorem~\ref{imaginary n-permutable 2}. First, we suppose that Figure~\ref{co-ops odd}, for $n$ odd, or Figure~\ref{co-ops even}, for $n$ even, represents a limit where $\beta_X$ is a regular epimorphism. For the limit of the outer diagram in Figure~\ref{ternary co-ops}, we want to prove that $\alpha_X$ is a regular epimorphism. As in Remark~\ref{(n+1)-ary vs ternary}, the regular epimorphism $\beta_X\colon \regepi{B(X)}{X}$ and the morphisms
\[
\left((\nabla_i)_X+1_X+(\nabla_{n-i})_X\right) {(v_i)}_X\colon B(X)\to 3X,\;\; i\in \{1, \dots, n-1\},
\]
give another cone on the outer diagram of Figure~\ref{ternary co-ops}. Here, once again, ${(\nabla_i)}_X=[1_X, \dots, 1_X ]^{\T}\colon iX\to X$. This guarantees the existence of a unique morphism $\lambda\colon\map{B(X)}{A(X)}$ such that, in particular, $\beta_X=\alpha_X\lambda$. Hence $\alpha_X$ is a regular epimorphism.

Conversely, suppose that Figure~\ref{ternary co-ops} represents a limit where $\alpha_X$ is a regular epimorphism. The regular epimorphism $\alpha_X\colon \regepi{A(X)}{X}$ and the morphisms
\[
 \threesom{\iota_i}{\iota_{i+1}}{\iota_{i+2}}{(w_i)}_X\colon A(X)\to (n+1)X,\;\; i\in \{1, \dots, n-1\},
\]
 for the coproduct inclusions $\iota_1, \dots, \iota_{n+1}\colon X\to (n+1)X$, give another cone on the outer diagram of Figure~\ref{co-ops odd}, for $n$ odd, or Figure~\ref{co-ops even}, for $n$ even. Consequently, $\beta_X$ is a regular epimorphism.
\end{proof}

We end with a categorical version of Theorem~\ref{3 of Hagemann-Mitschke}.

\begin{lemma}\label{(2n-2)-permutable} Let $\XX$ be a regular category such that, for some natural number $n\geq 2$, we have $E^{n}\leq E^{n-1}$ for every reflexive relation~$E$. Then $\XX$ is $(2n-2)$-permutable.
\end{lemma}
\begin{proof}
Let $R$ and $S$ be equivalence relations on a given object $Y$. For $E=SR$, we have $(S,R)_{2n}\leq (S,R)_{2n-2}$ by assumption. But
\[
(R,S)_{2n-2}\leq S(R,S)_{2n-2}R=(S,R)_{2n}\leq (S,R)_{2n-2},
\]
which proves our claim.
\end{proof}

\begin{theorem}\label{imaginary total thm}
Let $\XX$ be a regular category with binary coproducts. The following statements are equivalent:
\begin{enumerate}
\item $\XX$ is an $n$-permutable category;
\item for every reflexive relation $R$, we have $R^{\op}\leq R^{n-1}$;
\item for every reflexive relation $R$, we have $R^{n}\leq R^{n-1}$.
\end{enumerate}
\end{theorem}
\begin{proof}
(a) $\Rightarrow$ (b) Let $R$ be a reflexive relation on $Y$ and consider morphisms $x$, $y\colon\map{X}{Y}$ such that $\langle x,y\rangle\in R^{\op}$; hence $\langle y, x\rangle\in R$. Since $\XX$ is an $n$-permutable category, there exist approximate Hagemann--Mitschke co-operations $w_1$, \dots, $w_{n-1}$ with approximation $\alpha$. These are defined, for each object $X$ in $\XX$, as in Figure~\ref{ternary co-ops}, where $\alpha_X$ is a regular epimorphism. Since $R$ is a reflexive relation, we have $\langle x, x\rangle$, $\langle y,x\rangle$, $\langle y, y\rangle\in R$, so that also
\[
\left\langle\threesom{x}{y}{y},\threesom{x}{x}{y}\right\rangle\in R.
\]
Precomposing with each approximate co-operation, we get
\[
\begin{cases}
\left\langle x\alpha_X,\threesom{x}{x}{y} {(w_1)}_X \right\rangle\in R \vspace{3pt}\\
\left\langle \threesom{x}{y}{y} {(w_2)}_X ,\threesom{x}{x}{y} {(w_2)}_X \right\rangle\in R \vspace{3pt}\\
\qquad\vdots \\
\left\langle \threesom{x}{y}{y} {(w_{n-2})}_X ,\threesom{x}{x}{y} {(w_{n-2})}_X \right\rangle\in R \vspace{3pt}\\
\left\langle \threesom{x}{y}{y}{(w_{n-1})}_X, y\alpha_X \right\rangle\in R.
\end{cases}
\]
From $(\nabla_X+1_X){(w_j)}_X = (1_X+\nabla_X){(w_{j+1})}_X$, for $j\in \{1,\dots, n-2\}$, we can conclude that
\[
\langle x\alpha_X, y\alpha_X\rangle =\langle x,y\rangle\alpha_X\in R^{n-1}.
\]
So $\langle x,y\rangle\in R^{n-1}$, since $\alpha_X$ is a regular epimorphism.

(b) $\Rightarrow$ (a) For any object $X$ in $\XX$, consider the following reflexive graph and the reflexive relation $R$ on $2X$ which results by taking the (regular epi, mono) factorisation in
\[
\xymatrix@!0@R=5em@C=5em{
 3X\ar@<1.25ex>[rr]^-{\nabla_X+1_X}\ar@<-.75ex>[rr]_-{1_X+\nabla_X}\ar@{>>}[dr]_-{\pi} & & 2X.\ar@<-.25ex>[ll]\ar@<.5ex>[dl]\\
 & R\ar@<.5ex>[ur]^(.3){r_1}\ar@<-1.5ex>[ur]_(.3){r_2} }
\]
 From $\langle\iota_1,\iota_2\rangle\in R$ we get $\langle\iota_2,\iota_1\rangle\in R^{\op}\leq R^{n-1}$. So, there exists a regular epimorphism $\zeta\colon\regepi{Z}{X}$ together with morphisms $x_1$, \dots, $x_{n-2}\colon\map{Z}{2X}$ such that
 \[
 \langle\iota_2\zeta, x_1\rangle,\; \langle x_1,x_2 \rangle\;, \dots,\; \langle x_{n-3},x_{n-2}\rangle,\; \langle x_{n-2},\iota_1\zeta\rangle\in R.
\]
Let $k_1$, \dots, $k_{n-1}\colon\map{Z}{R}$ be the morphisms such that $\langle r_1, r_2\rangle k_1 =\langle\iota_2\zeta, x_1\rangle$, $\langle r_1, r_2\rangle k_i =\langle x_{i-1}, x_i\rangle$, $i\in \{2, \dots, n-2\}$, and $\langle r_1, r_2\rangle k_{n-1} =\langle x_{n-2},\iota_1\zeta\rangle$. From the pullback
\[
\xymatrix@!0@R=5em@C=12em{
 A(X)\ar[r]^-{\langle {(w_{n-1})}_X, \dots, {(w_1)}_X \rangle}\ar@{>>}[d]_-{\pi'} \pullback & (3X)^{n-1} \ar@{>>}[d]^-{\pi^{n-1}}\\
 Z\ar[r]_-{\langle k_1, \dots, k_{n-1} \rangle} & R^{n-1} }
\]
we get morphisms ${(w_1)}_X$, \dots, ${(w_{n-1})}_X$ and a regular epimorphism defined by $\alpha_X=\zeta\pi'$ such that the diagram in Figure~\ref{ternary co-ops} commutes. Then $\XX$ is an $n$-permutable category by Theorem~\ref{imaginary n-permutable 2}.

(a) $\Rightarrow$ (c) Let $R$ be a reflexive relation on $Y$ and consider morphisms $a$, $b\colon {X\to Y}$ such that $\langle a,b \rangle \in R^n$. Then there exists a regular epimorphism $\zeta\colon \regepi{Z}{X}$ together with morphisms $x_1, \dots, x_{n-1}\colon Z\to Y$ such that $\langle a\zeta, x_1 \rangle$, $\langle x_1,x_2 \rangle$, \dots, $\langle x_{n-2}, x_{n-1} \rangle$, $\langle x_{n-1}, b\zeta \rangle\in R$. Since $\XX$ is an $n$-permutable category, there are approximate Hagemann--Mitschke co-operations $w_1$, \dots, $w_{n-1}$ with approximation~$\alpha$ defined, for each object $X$ in $\XX$, as in Figure~\ref{ternary co-ops}, where $\alpha_X$ is a regular epimorphism. Since $R$ is a reflexive relation, we have
\[
\resizebox{\textwidth}{!}
{$\begin{aligned}
 \langle a\zeta, x_1 \rangle, \langle x_1, x_1 \rangle, \langle x_1,x_2 \rangle \in R
 &\Rightarrow \left \langle a\zeta \alpha_X, \threesom{x_1}{x_1}{x_2}{(w_1)}_X \right\rangle \in R, \vspace{3pt}\\
 \langle x_1,x_2 \rangle, \langle x_2, x_2 \rangle, \langle x_2,x_3 \rangle \in R
 &\Rightarrow \left \langle \threesom{x_1}{x_2}{x_2}{(w_2)}_X, \threesom{x_2}{x_2}{x_3}{(w_2)}_X \right\rangle \in R, \vspace{3pt}\\
&\;\;\vdots \\
 \langle x_{n-3},x_{n-2} \rangle, \langle x_{n-2}, x_{n-2} \rangle, \langle x_{n-2},x_{n-1} \rangle \in R
 &\Rightarrow \left \langle \threesom{x_{n-3}}{x_{n-2}}{x_{n-2}}{(w_{n-2})}_X, \threesom{x_{n-2}}{x_{n-2}}{x_{n-1}}{(w_{n-2})}_X \right\rangle \in R, \vspace{3pt}\\
 \langle x_{n-2},x_{n-1} \rangle, \langle x_{n-1}, x_{n-1} \rangle, \langle x_{n-1},b\zeta \rangle \in R
 &\Rightarrow \left \langle \threesom{x_{n-2}}{x_{n-1}}{x_{n-1}}{(w_{n-1})}_X, b\zeta \alpha_X \right\rangle \in R.
\end{aligned}$}
\]
From $(\nabla_X+1_X){(w_j)}_X=(1_X+\nabla_X){(w_{j+1})}_X$ for $j\in\{1, \dots, n-2\}$ we conclude that
\[
\langle a\zeta \alpha_X, b\zeta\alpha_X \rangle=\langle a, b \rangle \zeta \alpha_X \in R^{n-1},
\]
so $\langle a,b \rangle \in R^{n-1}$ since $\zeta$ and $\alpha_X$ are regular epimorphisms.

(c) $\Rightarrow$ (b) By Lemma~\ref{(2n-2)-permutable}, we know that $\XX$ is $(2n-2)$-permutable. Let $R$ be a reflexive relation. Using the equivalence (a) $\Leftrightarrow$ (b) for $(2n-2)$-permutability, we have $R^{\op}\leq R^{2n-3}$. Using our assumption $R^n\leq R^{n-1}$ several (in fact, $n-2$) times we obtain
\[
R^{\op}\leq R^{2n-3}\leq R^{2n-2}\leq \dots \leq R^n\leq R^{n-1}.
\]
This finishes the proof.
\end{proof}

\begin{corollary}\label{Corollary RTS}
In an $n$-permutable category with binary coproducts, any reflexive and transitive relation is symmetric.
\end{corollary}
\begin{proof}
It suffices to combine $R^{\op}\leq R^{n-1}$ for $R$ reflexive with $RR\leq R$ for $R$ transitive to see that $R^{\op}\leq R$.
\end{proof}

\section*{Acknowledgement}

Thanks to the referee for many detailed comments and important suggestions.


\providecommand{\noopsort}[1]{}
\providecommand{\bysame}{\leavevmode\hbox to3em{\hrulefill}\thinspace}
\providecommand{\MR}{\relax\ifhmode\unskip\space\fi MR }
\providecommand{\MRhref}[2]{%
  \href{http://www.ams.org/mathscinet-getitem?mr=#1}{#2}
}
\providecommand{\href}[2]{#2}

\end{document}